\documentclass{article}

\def\polhk#1{\setbox0=\hbox{#1}{\ooalign{\hidewidth
    \lower1.5ex\hbox{`}\hidewidth\crcr\unhbox0}}}

\usepackage{amsfonts}
\usepackage{amsmath}
\usepackage{amssymb}
\usepackage{amsthm}
\usepackage{ifthen}

\newtheorem{theorem}{Theorem}[section]
\newtheorem{proposition}[theorem]{Proposition}
\newtheorem{corollary}[theorem]{Corollary}
\newtheorem{lemma}[theorem]{Lemma}
\theoremstyle{remark}

\theoremstyle{definition}

\newcommand{\parder}[3][Default]{
	\frac{\partial \ifthenelse{\equal{#1}{Default}}{}{^{#1}}#2}{
              \partial #3 \ifthenelse{\equal{#1}{Default}}{}{^{#1}}}}

\newcommand{\jac}{{\mathcal J}}
\newcommand{\hess}{{\mathcal H}}
\newcommand{\grad}{\nabla}

\newcommand{\GL}{\operatorname{GL}}
\newcommand{\imp}{{\mathversion{bold}$\Rightarrow$} }
\newcommand{\C}{{\mathbb C}}

\newcommand{\N}{{\mathbb N}}

\newcommand{\tp}{^{\mathrm t}}
\newcommand{\I}{{\mathrm i}}

\newcommand{\trdeg}{\operatorname{trdeg}}

\newcommand{\chr}{\operatorname{chr}}

\makeatletter
\newcommand{\nolisttopbreak}{\vspace{\topsep}\nobreak\@afterheading}
\makeatother

\newenvironment{listproof}[1][\proofname]{\begin{proof}[#1]\mbox{}\nolisttopbreak}{\end{proof}}

\title{Irreducibility properties of Keller maps}
\author{
Michiel de Bondt\footnote{Supported by the Netherlands
                          Organisation for Scientific Research (NWO).} \\
Department of Mathematics, Radboud University \\
Nijmegen, The Netherlands \\
\emph{E-mail:} M.deBondt@math.ru.nl
\and
Dan Yan\footnote{Supported by the National Natural Science Foundation of China 
                 (Grant No.11371343)} \\
Department of Mathematics and Computer Science,\\
Hunan Normal University, Changsha 410006, China \\
\emph{E-mail:} yan-dan-hi@163.com
}

\begin{document}

\maketitle

\begin{abstract}
\noindent
J{\polhk{e}}drzejewicz showed that a polynomial map over a field of characteristic zero is invertible,
if and only if the corresponding endomorphism maps irreducible polynomials to irreducible
polynomials. Furthermore, he showed that a polynomial map over a field of characteristic
zero is a Keller map, if and only if the corresponding endomorphism maps irreducible polynomials
to square-free polynomials. We show that the latter endomorphism maps other square-free polynomials
to square-free polynomials as well.

In connection with the above classification of invertible polynomial maps and the Jacobian Conjecture,
we study irreducible properties of several types of Keller maps, to each of which the Jacobian
Conjecture can be reduced. Herewith, we generalize the result of Bakalarski, that the components of 
cubic homogeneous Keller maps with a symmetric Jacobian matrix (over $\C$ and hence any field of 
characteristic zero) are irreducible.

Furthermore, we show that the Jacobian Conjecture can even be reduced to any of
these types with the extra condition that each affinely linear combination of the components of the
polynomial map is irreducible. This is somewhat similar to reducing the planar Jacobian Conjecture to the
so-called (planar) weak Jacobian Conjecture by Kaliman.
\end{abstract}


\bigskip

\section{Introduction}

Throughout this paper, we will write $x$ for the $n$ indeterminates $x_1, x_2, \ldots, x_n$,
where $n \in \N$. In a similar manner, we will write $y$ for $y_1, y_2, \ldots, y_n$ and
$z$ for $z_1, z_2, \ldots, z_n$. $K$ always denotes a field of characteristic zero and $\bar{K}$
is the algebraic closure of $K$. Let $F = (F_1, F_2, \ldots, F_m) \in K[x]^m$.
Then $F$ corresponds to the polynomial map $K^n \ni v \mapsto F(v) \in K^m$.
Write $\jac F$ for the Jacobian of $F$ with respect to $x$, i.e.\@
$$
\jac F := \jac_x F := \left( \begin{array}{cccc}
\parder{}{x_1} F_1 & \parder{}{x_2} F_1 & \cdots & \parder{}{x_n} F_1 \\
\parder{}{x_1} F_2 & \parder{}{x_2} F_2 & \cdots & \parder{}{x_n} F_2 \\
\vdots & \vdots & & \vdots \\
\parder{}{x_1} F_m & \parder{}{x_2} F_m & \cdots & \parder{}{x_n} F_m
\end{array} \right).
$$
Let $M\tp$ denote the transpose of a matrix $M$.
For a single polynomial $f \in K[x]$, write $\grad f$ for the gradient of $f$ with respect to $x$,
i.e.\@
$$
\grad f := \grad_x f := (\jac_x f)\tp = \left( \begin{array}{c} \parder{}{x_1} f \\ \parder{}{x_2} f \\
                                             \vdots \\ \parder{}{x_n} f \end{array} \right).
$$
Additionally, write $\hess f$ for the Hessian of $f$ with respect to $x$, i.e.\@
$$
\hess f := \jac_x \big(\grad_x f\big) = \left( \begin{array}{cccc}
\parder{}{x_1}\parder{}{x_1} f & \parder{}{x_2}\parder{}{x_1} f & \cdots & \parder{}{x_n}\parder{}{x_1} f \\
\parder{}{x_1}\parder{}{x_2} f & \parder{}{x_2}\parder{}{x_2} f & \cdots & \parder{}{x_n}\parder{}{x_2} f \\
\vdots & \vdots & \ddots & \vdots \\
\parder{}{x_1}\parder{}{x_n} f & \parder{}{x_2}\parder{}{x_n} f & \cdots & \parder{}{x_n}\parder{}{x_n} f
\end{array} \right)
$$

We say that a polynomial map $F$ is \emph{invertible} if $F$ has a polynomial inverse.
So an invertible polynomial map is bijective. The converse holds if $K = \bar{K}$
(see \cite[Th.\@ 4.2.1]{MR1790619}), but not in general (see \cite[(1.1.36)]{MR1790619}).
The well-known Jacobian Conjecture (JC for short), raised by O.H. Keller in 1939 in \cite{MR1550818},
states that a polynomial map $F: K^n \rightarrow K^n$ is invertible if its Jacobian determinant
$\det \jac F$ is a nonzero constant. From \cite[Th.\@ 4.2.1]{MR1790619} for $K = \C$ and
\cite[Prop.\@ 1.1.12]{MR1790619}, one deduces that it suffices to prove that 
$F$ is injective in the definition of JC.

This conjecture has been attacked by many people from various
research fields and remains open even for $n=2$! (Of course, a positive answer is
obvious for $n=1$.) See \cite{MR0663785} and \cite{MR1790619} and the references therein for a wonderful
70-years history of this famous conjecture. The condition that $\det \jac F \in K^{*}$ is called the {\em
Keller condition} and polynomial maps that satisfy this condition are called {\em Keller maps}.

Among the vast interesting and valid results, one result obtained by S.S.S.\@ Wang in \cite{MR0585736}
in 1980 is that the JC holds for all polynomial maps of degree 2 in all dimensions.
Another result is the reduction to degree
3, due to H.\@ Bass, E.\@ Connell and D.\@ Wright in \cite{MR0663785} in 1982 and
A.\@ Yagzhev in \cite{MR0592226} in 1980, which asserts that the JC is true if it
holds for all polynomial maps $F = x + H$, such that $H$ is cubic homogeneous,
i.e.\@ each component $H_i$ of $H$ is either zero or a cubic form.

Upon this reduction to the cubic homogeneous case, there are two subsequent reductions, but they
cannot be applied both. The first one is that additionally, $H_i$ is a third power of a linear form
for each $i$, see \cite{MR0714105}. The second one, which requires that the imaginary unit $\I \in K$,
is that $\jac H$ or equivalently $\jac F$ is symmetric,
see \cite{MR2138860}. By a special case of Poincar{\'e}'s lemma, this is the same as that $F = \grad f$
and $H = \grad h$ for certain polynomials $f, h \in K[x]$. If both $H_i$ is a (third) power of a linear form
and $H_i = \parder{}{x_i} h$ for each $i$, then $F = x + H$ is tame with inverse $x - H$, see
\cite{MR2183036} and \cite[Th.\@ 3.4]{MR2208537}.

In \cite[Th.\@ 3.7]{MR2333454}, S. Bakalarski proved the following interesting connection
between invertible polynomial maps and irreducibility: a Keller map from $\C^n$ to $\C^n$ is invertible,
if and only if the corresponding endomorphism maps irreducible polynomials to irreducible polynomials.
K. Rusek improved this result by showing that the Keller condition is not necessary.
This improved result was generalized to arbitrary fields of characteristic zero
in \cite[Th.\@ 5.2]{MR2965914} by P. J{\polhk{e}}drzejewicz. In \cite[Th.\@ 5.1]{MR2965914},
J{\polhk{e}}drzejewicz proved the following counterpart of this result:
a polynomial map from $K^n$ to $K^n$ is a Keller map,
if and only if the corresponding endomorphism maps irreducible polynomials to square-free polynomials.
We shall show in the next section that for Keller maps, the corresponding endomorphism even maps all
square-free polynomials to square-free polynomials.

In \cite{MR1106179}, S. Kaliman showed that in order to prove the JC in
dimension $n = 2$ for $K = \C$ (and hence for all $K$ by \cite[Prop.\@ 1.1.12]{MR1790619}),
one may assume that $F_1 + c$ is irreducible for every $c \in K$.
To prove the JC in dimension $n \ge 3$ for $K = \C$
one may even assume that $F_i + c$ is irreducible for
every $i \le n$ and $c \in K$. This was proved in \cite[Th.\@ 3]{MR1886939}.
We shall show that for the JC for all $n$, one may even assume that every affinely linear
combination of the components of $F$ is irreducible. Furthermore, we combine this reduction with
several other reductions of the JC, including both reductions in the previous
paragraph. See Theorems \ref{chKeller} and \ref{dKeller}.

In \cite{MR2607856}, Bakalarski proved that each component of $F$ is irreducible if
$F = x + H$, $\det \jac F = 1$, $H$ is cubic homogeneous and $\jac H$ is symmetric. 
(Actually, Bakalarski proved his result only for $K = \C$, but using Lefschetz' principle,
one may assume that $K \subseteq \C$, which gives the general case.) Notice that $F_i$ is
the image of $x_i$ under the corresponding endomorphism of $F$. We will generalize this result in
i) of Theorem \ref{bakext}, where we show that $F_i$ is irreducible if $\jac F$ is symmetric,
$\det \jac F \in K^{*}$ and $F_i = l + h$ such that $h$ and $l$ are homogeneous and
$\parder{}{x_i} l \in K^{*}$.
Notice that, as opposed to the result of Bakalarski, the index variable $i$ is free instead of
bound by an universal quantifier. So the conditions on $l$ and $h$ are for $F_i$ only, and not for the
$F_j$ with $j \ne i$.

Additionally, we show in Corollary \ref{irredcor} that $F_i$ is
irreducible if $\det \jac F \in K^{*}$ and the set of degrees of monomials of $F_i$ is $\{0,1,3\}$.
If we combine this result with the above-mentioned result of Theorem \ref{bakext}, we can conclude that
$F_i + c$ is irreducible for all $c \in K$ if $\jac F$ is symmetric,
$\det \jac F \in K^{*}$, and $F = l + h$ such that $h$ is cubic homogeneous, $\deg l = 1$,
and $\parder{}{x_i} l \in K^{*}$. The latter result can also be found in i) of Theorem \ref{bakext}.

As an end of this introduction, we summarize some results in connection with coordinates.
A polynomial $f \in K[x]$ is a {\em coordinate} if there exists
an invertible polynomial map $F \in K[x]^n$ such that $f = F_1$. After some partial results in
\cite{MR1456093} and \cite{MR1679074}, Z. Jelonek proved in \cite{MR2905013} that a polynomial map
over $K$ is invertible, if and only if the corresponding endomorphism maps coordinates to coordinates.
The result of \cite[Lm.\@ 2.3]{MR1456093} by H. Derksen is that a polynomial map over
$\bar{K}$ is a Keller map, if and only if the corresponding endomorphism maps linear
coordinates to polynomials with nowhere vanishing gradients (for instance coordinates).

Obviously, Derksen's result is still valid if we replace `linear coordinates' by `coordinates'.
It is however not true in general that a polynomial
map over $K$ is invertible, if and only if the corresponding endomorphism maps {\em linear} coordinates to
coordinates, see \cite[Th.\@ 2.1]{MR1682757} (so Derksen's result is only valid if $K = \bar{K}$).
But  C. Cheng and A. van den Essen proved in \cite[Th.\@ 1.1]{MR1653437}, that in the case $n=2$,
it indeed suffices to show that the images of linear coordinates are coordinates.
Furthermore, A. van den Essen and V. Shpilrain
showed in \cite[Th.\@ 1.2]{MR1456093} that Keller maps $F$ are invertible if
$F_1$ is a coordinate and the JC holds in dimension $n - 1$.

\section{Some properties of Keller maps}

We start with a generalization of \cite[Th.\@ 4.1]{MR2965914} by J{\polhk{e}}drzejewicz.
To be precise, \cite[Th.\@ 4.1]{MR2965914} is the equivalence of 1) and 2) in the theorem
below, for the case that $g$ is irreducible.

\begin{theorem} \label{J4.1gen}
Let $F \in K[x]^n$
be an arbitrary polynomial map. If $g \in K[x]$ is {\em square-free}, then
the following conditions are equivalent:
\begin{enumerate}

\item[1)] $g \mid \det \jac F$,

\item[2)] for every {\em irreducible} $\tilde{g} \mid g$, there exists an {\em irreducible} polynomial
          $\tilde{w} \in K[y]$ such that $\tilde{g}^2 \mid \tilde{w}(F)$,

\item[3)] $g^2 \mid w(F)$ for some {\em square-free} polynomial $w \in K[y]$.

\end{enumerate}
\end{theorem}

\begin{proof}
Assume that $g \in K[x]$ is square-free. The equivalence of 1) and 2) follows by applying
\cite[Th.\@ 4.1]{MR2965914} for all irreducible polynomials $\tilde{g} \mid g$.
To prove 2) $\Longrightarrow$ 3), take for $w$ in 3) the least common multiple of all
$\tilde{w}$ appearing in 2). Then $w(F)$ in 3) is a common multiple of the $\tilde{w}(F)$
appearing in 2). Since $g^2$ in 3) is the least common multiple of the $\tilde{g}^2$ appearing
in 2), 2) $\Longrightarrow$ 3) follows. Hence it remains to show 3) $\Longrightarrow$ 2).

So assume 3) and let $\tilde{g}$ be an arbitrary irreducible divisor of $g$. We have to show that there
exists an irreducible $\tilde{w} \in K[y]$ such that $\tilde{g}^2 \mid \tilde{w}(F)$.
Since $g^2 \mid w(F)$, we can
decompose $w = w_1 w_2$, such that $w_1$ is irreducible and $\tilde{g} \mid w_1(F)$.
If $\tilde{g}^2 \mid w_1(F)$, then we are done, so suppose that $\tilde{g}^2 \nmid w_1(F)$. Then
$\tilde{g} \mid w_2(F)$. Let $\bar{F}$ be the residue classes of $F$ modulo $\tilde{g}$, i.e.\@
$\bar{F}_i = F_i + (\tilde{g})$ for each $i$. Define $r := \trdeg_K K(\bar{F}_1,\bar{F}_2,\ldots,\bar{F}_n)$
and assume without loss of generality that $\bar{F}_1,\bar{F}_2,\ldots,\bar{F}_r$ are algebraically
independent over $K$. Then $r \le n-1$, because $w_1(\bar{F}) = 0$.

If $r \le n-2$, then we can follow the last paragraph in the proof of (i) $\Longrightarrow$ (ii)
of \cite[Th.\@ 4.1]{MR2965914} verbatim to obtain that $\tilde{g}^2 \mid \tilde{w}(F)$ for some
irreducible $\tilde{w} \in K[y]$. So assume that
$r = n-1$. Notice that $w_1$ and $w_2$ are relatively prime, because $w$ is square-free.
Hence the ideal $(w_1,w_2)$ is not contained in a principal prime ideal of $K[y]$.
Since $K[y]$ is a unique factorization domain, we can deduce that the ideal $(w_1,w_2)$ has height at least two.
On the other hand, the ideal in $K[y]$ of algebraic relations between $\bar{F}_1,\bar{F}_2,
\ldots,\bar{F}_n$ has height $n - r$ and contains $(w_1,w_2)$. So $n - r \ge 2$, which contradicts $r = n-1$.
\end{proof}

\noindent
One can also obtain a contradiction to $r = n-1$ by showing that the resultant with respect to
$y_n$ of $w_1$ and $w_2$ is a nontrivial algebraic relation between $\bar{F}_1,\bar{F}_2,\ldots,
\bar{F}_{n-1}$.

Just like \cite[Th.\@ 4.1]{MR2965914}, its immediate consequence \cite[Cor.\@ 4.2]{MR2965914}
can be generalized. We do this by extending it with one line, namely property 3).

\begin{corollary}
Let $F \in K[x]^n$
be an arbitrary polynomial map. Then the following conditions are equivalent:
\begin{enumerate}

\item[1)] $\det \jac F \in K^{*}$,

\item[2)] for every {\em irreducible} polynomial $w \in K[y]$, the polynomial $w(F)$ is square-free,

\item[3)] for every {\em square-free} polynomial $w \in K[y]$, the polynomial $w(F)$ is square-free.

\end{enumerate}
\end{corollary}

\begin{proof}
Every assertion is equivalent to the nonexistence of an irreducible $g$ in the respective
assertion of Theorem \ref{J4.1gen}.
\end{proof}

\noindent
We end this section with a theorem about some reducibility properties which cannot be combined with the
Keller condition. We use a result of \cite{MR1285548} for that.

\begin{theorem} \label{lindiv}
Let $F \in K[x]^n$
be a Keller map. Suppose that for each $i$, $F_i$ is of the form $L_i H_i$, where $\deg L_i = 1$.
Then the following statements are equivalent.
\begin{enumerate}

\item[1)] The linear part of $L_i(0) H_i$ is divisible by $L_i - L_i(0)$ for each $i$,

\item[2)] $L(0)$ is contained in the column space of $\jac L$,

\item[3)] $\det \jac L \in K^{*}$,

\item[4)] $\deg F = 1$,

\item[5)] $F_i$ is irreducible for each $i$.

\end{enumerate}
\end{theorem}

\begin{proof}
Notice that 3) $\Longrightarrow$ 2) and 4) $\Longleftrightarrow$ 5) are trivial.
Hence it suffices to prove the following.
\begin{description}

\item[1) \imp 3)]
Suppose that 1) holds. Since $F_i = L_i(0) H_i + \big(L_i - L_i(0)\big) H_i$ for all $i$,
we can deduce that for each $i$, the linear part of $F_i$ is equal to $\big(L_i - L_i(0)\big) c_i$ for
some $c_i \in K$. Hence $c_1 c_2 \cdots c_n \det \jac L = (\det \jac F)|_{x=0}$.
Now 3) follows from the Keller condition on $F$.

\item[2) \imp 3)]
Let $*$ denote the Hadamard product and suppose that 2) holds. Say that $\jac L \cdot a = L(0)$,
where $a \in K^n$. Then the constant part of $L(x-a)$ is equal to
$$
L(-a) = \jac L|_{x=-a} \cdot (-a) + L(0) = \jac L \cdot (-a) + \jac L \cdot a = 0.
$$
Hence the linear part of $F(x-a) = L(x-a) * H(x-a)$ is equal to $L(x-a) * H(-a)$.
Using the Keller condition for $F$, $\det \jac L \mid \det (\jac F)|_{x=x-a} \in K^{*}$
follows, which is 3).

\item[3) \imp 4)]
Suppose that 3) holds. Then $L$ is invertible and
$$
F\big(L^{-1}(x)\big) = L\big(L^{-1}(x)\big) * H\big(L^{-1}(x)\big) = x * H\big(L^{-1}(x)\big)
$$
is a Keller map as well. It follows from \cite[Prop.\@ 6]{MR1285548} that 
$\deg H = \deg H\big(L^{-1}(x)\big) = 0$. Hence $\deg F = 1$.

\item[4) \imp 1)]
Suppose that $\deg F = 1$. Then $\deg H_i=0$ and hence $L_i(0)H_i\in K$ for each $i$.
Thus for each $i$, the linear part of $L_i(0)H_i$ is zero, which is divisible by $L_i - L_i(0)$.
\qedhere

\end{description}
\end{proof}

\section{Irreducibility results for reductions of the JC}

\begin{theorem} \label{chKeller}
Assume $F \in K[x]^n$ is a cubic Keller map without quadratic part.
Then there exists a $\lambda \in K^n$ such that for
\begin{align}
G &= (F - \lambda x_{n+1}^3, x_{n+1}), \label{ch} \\
G &= (F - \lambda x_{n+1}^3, x_{n+1}, x_{n+2} + x_{n+1}^3), \label{dz} \\
\intertext{and}
G &= (F - \lambda x_{n+1}^3, x_{n+2} - 3 x \tp \lambda x_{n+1}^2, x_{n+1}) \nonumber \\
  &= (F,0,0) + \grad_{x,x_{n+1},x_{n+2}} \big(x_{n+1}x_{n+2} - x\tp \lambda x_{n+1}^3\big), \label{sch}
\end{align}
every linear combination of the components of $G$ and $1$ which is reducible is already a
linear combination of $1$.

Furthermore, $G$ is a cubic Keller map without quadratic part, and
$F$ is invertible, if and only if $G$ is invertible. Additionally, we have the following.
\begin{enumerate}

\item[i)] If $F$ is linearly conjugate to a Dru{\.z}kowski map, then so is $G$ in \eqref{dz}.

\item[ii)] If $\jac F$ is symmetric, then so is $\jac_{x,x_{n+1},x_{n+2}} G$ in \eqref{sch}.

\end{enumerate}
\end{theorem}

\begin{proof}
In Corollary \ref{lambdacol}, we will prove the first claim (the existence of $\lambda$)
for
\begin {equation} \label{gh}
G = (F - \lambda x_{n+1}^3, x_{n+1}, x_{n+2} - h),
\end{equation}
where $h \in K[x,x_{n+1}]$ is arbitrary. This immediately gives the first claim for $G$ in
\eqref{dz}. To obtain the first claim for $G$ in \eqref{ch} and \eqref{sch},
we remove the last component of $G$ and interchange the last two
components of $G$ respectively in \eqref{gh}. Thus a $\lambda \in K^n$ as given exists.

By expansion of the determinant along the $(n+2)$-th column, if present, and subsequently
along the last row, we see that
\begin{align*}
\det \jac_{x,x_{n+1}} G &= \det \jac F  \mbox{ in \eqref{ch},} \\
\det \jac_{x,x_{n+1},x_{n+2}} G &= \det \jac F \mbox{ in \eqref{dz}, and} \\
\det \jac_{x,x_{n+1},x_{n+2}} G &= -\det \jac F \mbox{ in \eqref{sch}.}
\end{align*}
Hence $G$ is a Keller map, and one can easily verify that $G$ is a cubic Keller map
without quadratic part.

We only prove the rest of this theorem for the cases \eqref{dz} and \eqref{sch},
since the case \eqref{ch} is similar. Let $E = (x - \lambda x_{n+1}^3, x_{n+1}, x_{n+2})$.
\begin{enumerate}

\item[i)] Assume that $G$ is as in \eqref{dz}. Then
          $$
          G = E(F,x_{n+1},x_{n+2})|_{x_{n+2} = x_{n+2} + x_{n+1}^3}.
          $$
          Consequently, $F$ is invertible, if and only if $G$ is invertible.

          Suppose that $TF(T^{-1}x)$ is a Dru{\.z}kowski map. Set
          $$
          \tilde{T} = \left( \begin{array}{ccc|cc}
                      & & & 0 &  \\
                      & T & & \vdots & \!\!\!\!T\lambda\!\!\!\! \\
                      & & & 0 &  \\ \hline
                      0 & \cdots & 0 & 1 & 0 \\
                      0 & \cdots & 0 & 0 & 1
                      \end{array} \right).
          $$
          Then $\tilde{T}G\big(\tilde{T}^{-1}(x,x_{n+1},x_{n+2})\big)$ is
          a Dru{\.z}kowski map as well. Hence $G$ is linearly conjugate to a
          Dru{\.z}kowski map if $F$ is.

\item[ii)] Assume that $G$ is as in \eqref{sch}. Then
           $$
           (G_1,G_2,\ldots,G_n,G_{n+2},G_{n+1})
           = E(F,x_{n+1},x_{n+2})|_{x_{n+2} = x_{n+2} - 3 x\tp \lambda x_{n+1}^2}.
           $$
           Consequently, $F$ is invertible, if and only if $G$ is invertible.

           Suppose that $\jac F$ is symmetric. Then by \eqref{sch}, $\jac G$ is symmetric as well,
           which completes the proof. \qedhere

\end{enumerate}
\end{proof}

\begin{theorem} \label{dKeller}
Assume that $F \in K[x]^n$ is a Keller map and let $d \ge 2$ be an integer. 
Then for
\begin{align}
G &= (F - y^{*d}, y), \label{chl} \\
G &= (F - y^{*d}, y, z + y^{*d}), \label{dzl} \\
\intertext{and}
G &= (F - y^{*d}, z - d x\tp y^{*(d-1)}, y) \nonumber \\
  &= (F,0,\ldots,0,0,\ldots,0) + \grad_{x,y,z} \big(y\tp z - x\tp y^{*d}\big), \label{schl}
\end{align}
every linear combination of the components of $G$ and $1$ which is reducible is already a
linear combination of $1$.

Furthermore, $G$ is a Keller map and $F$ is invertible, if and only if $G$ is invertible,
and we have the following.
\begin{enumerate}

\item[i)] If $F - x$ is linearly conjugate to a power linear map of degree $d$,
then so is $G-(x,y,z)$ in \eqref{dzl}.

\item[ii)] If $\jac F$ is symmetric, then so is $\jac_{x,y,z} G$ in \eqref{schl}.

\end{enumerate}
\end{theorem}

\begin{proof}
In Corollary \ref{mulem2col}, we will prove the first claim for
\begin{equation} \label{ghl}
G = (F - y^{*d}, y, z - H),
\end{equation}
where $H \in K[x,y]^n$ is arbitrary. The proof of this theorem is a multidimensional variation
of the proof of Theorem \ref{chKeller}, where Corollary \ref{mulem2col} is used by way of
\eqref{ghl} instead of Corollary \ref{lambdacol} by way of \eqref{gh}.
\end{proof}

\noindent
i) of Theorem \ref{bakext} below is a generalization of \cite[Th.\@ 2.2]{MR2607856} by Bakalarski.
In ii) of Theorem \ref{bakext}, \eqref{symred} with $u = \I$ and $u' = -\I$ corresponds to the
gradient reduction of the JC in \cite{MR2138860}. Indeed, if $F = x + H$, $f = - x\tp \cdot H(x)$,
and $f_H$ is as in Equation (3) of \cite{MR2138860} (i.e.\@ $f_H = -i y\tp \cdot H(x+\I y)$), then
\begin{align*}
\frac12 \sum_{i=1}^n x_i^2 &+ \frac12 \sum_{i=1}^n y_i^2 + f_H \\
&= \frac12 (x - \I y)\tp \cdot (x + \I y) - i y\tp \cdot H(x+\I y) \\
&= \frac12 (x - \I y)\tp \cdot \big( (x + \I y) + H(x + \I y) \big)
   - \frac12 (x + \I y)\tp \cdot H(x + \I y) \\
&= \frac12 (x - \I y) F(x+\I y) + \frac12 f(x + \I y),
\end{align*}
of which the gradient map is $\frac12 G$, where $G$ is as
in \eqref{symred} with $u = \I$ and $u' = -\I$. 

Taking $u = 1$ and $u' = -1$ in \eqref{symred} gives a
gradient reduction of the JC that does not require imaginary units, and $-\frac12 G$ has 
linear part $(-x,y)$ in that case, provided $F$ has linear part $x$ and $f$ has a trivial 
quadratic part. This linear part is also the result of the gradient reduction 
of the JC in \cite[Th.\@ 3.1 (i)]{MR2208537} (for cubic homogeneous maps). 
The reduction in \cite[Th.\@ 3.1 (i)]{MR2208537} is 
however slightly different, because it still requires a certain square root in $K$, 
namely $\sqrt{2}$.

\begin{theorem} \label{bakext}
Assume that $G$ is a Keller map with a symmetric Jacobian over $K$.
Write $G_i^{(1)}$ for the linear part of $G_i$.
\begin{enumerate}

\item[i)] If $G \in K[x]^n$, then $G_i$ (${+}\;c$) is irreducible (for all $c \in K$) if
          $\parder{}{x_i} G_i^{(1)} \ne 0$ and $G_i - G_i^{(1)}$ is (cubic) homogeneous.

\item[ii)] If $G \in K[x,y]^{2n}$, then $G_i$ is irreducible for all $i$ if
           $G$ is of the form
           \begin{equation}
           G = \grad_{x,y} \big(f(x+uy) + (x+u'y)\tp F(x+uy)\big), \label{symred}
           \end{equation}
           where $u,u' \in K$ such that $u \ne 0$, $f \in K[x]$ and $F_i \in K[x]$ for all $i \le n$.
           Furthermore, $u \ne u'$ and $F$ is a Keller map in this case,
           and additionally $F$ is invertible, if and only if $G$ is.

\end{enumerate}
\end{theorem}

\begin{listproof} 
\begin{enumerate}

\item[i)] Suppose that $G \in K[x]^n$ and $c \in K$ such that $G_i + c$ is reducible,
$\parder{}{x_i} G_i^{(1)} \ne 0$ and $G_i - G_i^{(1)}$ is homogeneous of degree $d$. Then $d \ge 2$.
In order to prove i) both with and without the parenthesized parts, it suffices
to obtain a contradiction in the case where either $d \le 3$ or $c = 0$. We shall derive a contradiction
by showing that $\deg G_i = 1$ in both subcases.

Since $G$ is a Keller map,
we see by expansion of the Jacobian determinant of $G$ along the $i$-th row that $\jac G_i$
is unimodular.
In 1) $\Longrightarrow$ 3) of Corollary \ref{irredcor}, we will show that
$(G_i^{(1)})^2 \mid G_i - G_i^{(1)}$ if $G_i - G_i^{(1)}$ is homogeneous of degree
$d \ge 2$ and $\jac G_i$ is unimodular, and either $d \le 3$ or $c = 0$.
In Corollary \ref{symdiagcol}, we will show that $\deg G_i = 1$ if
$\parder{}{x_i} G_i^{(1)} \ne 0$, $(G_i^{(1)})^2 \mid G_i - G_i^{(1)}$ and
$G$ is a Keller map with a symmetric Jacobian. Hence $\deg G_i = 1$.

\item[ii)] Suppose that $G \in K[x,y]^{2n}$ is as in \eqref{symred}.
We can rewrite \eqref{symred} as $G = \grad_{x,y} \big((f + y\tp F)|_{(x,y) = (x + u y, x + u' y)} \big)$.
Hence the chain rule for $\grad_{x,y} = \jac_{x,y} \tp$ tells us that
$$
G = \left( \begin{array}{cc} I_n & I_n \\  u I_n & u' I_n \end{array} \right)
    \big(\grad_{x,y} (f + y\tp F)\big)\big|_{(x,y) = (x + u y, x + u' y)}.
$$
In Lemma \ref{symm}, we will show that $F$ is invertible, if and only if $\big(\grad_{x,y} (f + y\tp F)\big)$
is, and that $\det \hess_{x,y} (f + y\tp F) = (-1)^n (\det \jac F)^2$.
Since $G$ is a Keller map, we see that $u \ne u'$ and that $F$ is a Keller map.

Again in Lemma \ref{symm}, we will show that $\mu\tp \grad_{x,y} (f + y\tp F)$ is irreducible
for all $\mu \in K^{2n}$ such that $\mu_i \ne 0$ for some $i \le n$, provided $F$ is a Keller map.
Thus $G_i$ is irreducible for all $i \le n$. By assumption of $u \ne 0$,
$G_i$ is irreducible for all $i > n$ as well. \qedhere

\end{enumerate}
\end{listproof}

\noindent
We end this section by showing that for components $F_i$ of power linear Keller maps over $K$,
$F_i - c$ is irreducible for all $c \in K$.

\begin{proposition} \label{prop1}
Assume that $f \in K[x]$, such that $\deg f \ne 1$ and
$h := f - x_i$ is a polynomial in a linear form. If $h \in K[x_i]$, then $f$ is not a component
of a Keller map. If $h \notin K[x_i]$, then $f$ is a tame coordinate.
\end{proposition}

\begin{proof}
If $h \in K[x_i]$, then $f$ cannot be a component of a Keller map because $\jac f$ is not unimodular.
So assume $h \notin K[x_i]$. Then there exists a
$T \in \GL_n(K)$ such that $T_1 x = x_i$ and $h \in K[T_2 x]$, say that $h = p(T_2 x)$.
It follows that $f$ is the first component
of the composition of the elementary invertible map $(x_1 + p(x_2), x_2, x_3, \ldots, x_n)$ and $T$.
\end{proof}

\noindent
Notice that i) of Theorem \ref{bakext} and Proposition \ref{prop1} are results about Keller maps of
some type, rather than subsequent reductions to obtain irreducibility results. More results of this type
are Theorem \ref{irredlc} and its corollary, and Theorems \ref{irredth} and \ref{symth}.

\section{Irreducibility lemmas for polynomials with unimodular gradients} \label{irlemu}

\begin{lemma} \label{nonsingular}
Let $f \in \bar{K}[x]$ such that $\deg f = d \ge 2$. Suppose that $f$ has monomials of degrees 
$0, 1, d$ only, and that $x_1 - c \mid f$ for some $c \in \bar{K}^{*}$. If $f$ is nonsingular, then
$$
f = c'(x_1^d - c^d) + (c'' - c' d c^{d-1})(x_1 - c)
$$
for some $c',c'' \in \bar{K}^{*}$.
\end{lemma}

\begin{proof}
Since $f$ is nonsingular, we have 
$f = \big(g\cdot(x_1 - c) + c''\big) \cdot (x_1 - c)$ 
for some $g \in \bar{K}[x]$ and $c'' \in \bar{K}^{*}$. Hence 
$$
h := g \cdot (x_1-c)^2 = f - c''(x_1-c)
$$ 
has monomials of degrees $0, 1, d$ only. Let $h'$ be the derivative of $h$ with respect to $x_1$, and take
$c' \in \bar{K}$ such that the constant part of $h'$ is equal to $-c' d c^{d-1}$ (if $h'$ has one,
otherwise take $c' = 0$). Since $h'$ has only monomials of degree $0$ and $d-1$, we deduce from 
$x_1 - c \mid h'$ that $h'$ is completely determined by its constant part. More precisely, 
$$
h' = c' d (x_1^{d-1} - c^{d-1}), \qquad \mbox{so} \qquad
h = c' x_1^d - c' d c^{d-1} x_1 + h|_{x_1 = 0}.
$$
Since $x_1 - c \mid h$ as well, it follows that $h|_{x_1 = 0}$ is completely determined by the 
other monomials of $h$, i.e.\@ the monomials of $h$ whose degree with respect to $x_1$ is positive. 
More precisely,
$$
h =  c' (x_1^d - c^d) - c' d c^{d-1}(x_1 - c).
$$
By definition of $h$, $f$ is as claimed, where $c' \ne 0$ because $\deg f \ge 2$.
\end{proof}

\noindent
From now on in this section, we shall write $f^{(k)}$ and $g^{(k)}$ for the homogeneous part of degree $k$ of
$f$ and $g$ respectively, or zero if $f$ or $g$ respectively has no such part.

\begin{lemma} \label{x_1|flem}
Let $f \in K[x]$ be nonzero. Assume that $f = gh$ is a
polynomial decomposition, such that $h(0) \ne 0$. Take $g^{*} \in K[x]$.

If $g^{*} \mid f^{(0)}, f^{(1)}, \ldots, f^{(\deg g)}$, then $g^{*} \mid g$.
\end{lemma}

\begin{proof}
Notice that $g^{*} \mid f^{(0)} = g^{(0)} h(0) \mid g^{(0)}$.
Suppose that $g^{*} \mid g^{(0)}, g^{(1)}, \ldots, \allowbreak g^{(i)}$ for some $i < \deg g$. Since
$g^{*} \mid f^{(i+1)}$, we obtain by expressing $f^{(i+1)}$ in the homogeneous parts of
$g$ and $h$ that $g^{*} \mid g^{(i+1)} h(0) \mid g^{(i+1)}$.
By induction on $i$, $g^{*} \mid g^{(0)}, g^{(1)}, \ldots, g^{(\deg g)}$,
so $g^{*} \mid g^{(0)} + g^{(1)} + \cdots + g^{(\deg g)} = g$.
\end{proof}

\begin{corollary} \label{x_1|f}
Assume $f \in K[x]$ such that $f - a x_1$ is
homogeneous of degree $d \ge 2$ for some nonzero $a \in K$. If $f$ is reducible, then
$x_1 \mid f$.
\end{corollary}

\begin{proof}
Suppose that $f$ is reducible. Then we can decompose $f = gh$ such that $h(0) \ne 0$ and
$\deg g \le d-1$. From Lemma \ref{x_1|flem} with $g^{*} = x_1$, we obtain that $x_1 \mid g \mid f$.
\end{proof}

\begin{corollary} \label{irredcor}
Assume that
$f \in K[x]$ has monomials of degree $0,1,d$ only, where $d \ge 2$,
say that $f = f^{(0)} + f^{(1)} + f^{(d)}$. Suppose that $f$ is reducible and $\jac f$ is unimodular. 
Then for
\begin{enumerate}

\item[1)] $d \le 3$ or $f^{(0)} = 0$,

\item[2)] $f$ has a divisor of degree $1$,

\item[3)] $f^{(0)} = 0$, $f^{(1)} \mid f$ and $(f^{(1)})^2 \mid f^{(d)}$,

\end{enumerate}
we have 1) $\Longrightarrow$ 2)  $\Longrightarrow$ 3).
\end{corollary}

\begin{proof}
Since $\jac f$ is unimodular, we have $f^{(1)} \ne 0$ and $f$ is nonsingular over $\bar{K}$.
Furthermore, $\deg f \ge 2$ because $f$ is reducible.

Since $2 \le \deg f \le d$, the case $d \le 3$ of 1) $\Longrightarrow$ 2) follows directly from the 
supposition that $f$ is reducible and the case $f^{(0)} = 0$ of 1) $\Longrightarrow$ 2) follows
from Corollary \ref{x_1|f}, because we may assume without loss of generality that $f^{(1)} = x_1$.

In order to prove 2)  $\Longrightarrow$ 3), suppose that $f$ has a divisor of degree $1$.
Without loss of generality, we may assume that
$x_1 - c \mid f$ for some $c \in K$. From $\deg f \ge 2$, we can subsequently deduce that
$(\parder{}{x_1} f~0~\cdots~0)$ is not unimodular. If $c \ne 0$, then
Lemma \ref{nonsingular} tells us that $\jac f = (\parder{}{x_1} f~0~\cdots~0)$ indeed, 
which contradicts that $\jac f$ is unimodular. So $c = 0$ and $x_1 \mid f$.
By the nonsingularity of $f$ over $\bar{K}$, we obtain that
$f = c'x_1(g x_1+1)$ for some $c' \in K^{*}$ and $g \in K[x]$. This gives 3).
\end{proof}

\begin{lemma} \label{irred}
Assume that $f \in \bar{K}[x]$
has degree at most $d$, and $g \in \bar{K}[y] \setminus \bar{K}$. Then for
\begin{enumerate}

\item[1)] $f - g$ is reducible,

\item[2)] $f - c$ is reducible for at least $d^2$ values of $c \in \bar{K}$,

\item[3)] $f \in \bar{K}[p^2,p^3]$ for some $p \in \bar{K}[x]$,

\item[4)] $\gcd\big\{\parder{}{x_1} f, \parder{}{x_2} f, \ldots, \parder{}{x_n} f\big\} \notin \bar{K}^{*}$,
\end{enumerate}
we have 1) $\Longrightarrow$ 2) $\Longrightarrow$ 3) $\Longrightarrow$ 4).
\end{lemma}

\begin{proof}
If $f \in \bar{K}$, then 2), 3) and 4) are trivially satisfied, so assume that $f \notin \bar{K}$.
\begin{description}

\item[3) \imp 4)]
Suppose that 3) holds. Then $0 < \deg p < \deg f$. Since $\parder{f}{p} \notin \bar{K}$
and $\parder{f}{p} \mid \parder{f}{x_i}$ for all $i$, we obtain 4).

\item[2) \imp 3)]
Suppose that 2) holds. From Corollary 3 of \cite[Th.\@ 37]{MR1770638}, it follows that
$f = g(p)$ for some $g \in \bar{K}[t]$ of degree $\ge 2$ and some $p \in \bar{K}[x]$.
Let $c$ be a root of $g'$. If we replace $g(t)$ by $g(t + c)$ and $p$ by $p - c$, then we
still have $f = g(p)$, but $g'(0)$ becomes $0$. Hence $g$ has no linear part any more, and 3) follows.

\item[1) \imp 2)]
Suppose that we have a decomposition $f - g = h_1 \cdot h_2$ over $\bar{K}$, 
where $\deg h_i \ge 1$ for both $i$.
If $h_1 \in \bar{K}[y]$, then the leading homogeneous part with respect to $x$ of $f - g$ is 
divisible by $h_1$. This contradicts that $f - g$ has no monomials with variables of both $x$ and $y$, 
so $\deg_x h_1 > 0$ and $\deg_x h_2 < \deg_x f$. 

Consequently, $\deg_x h_2 < \deg (f-c)$ and similarly
$\deg_x h_1 < \deg (f-c)$. For every $c \in \bar{K}$, $g = c$ has a solution $a \in \bar{K}^n$.
Now $h_1|_{y=a} \cdot h_2|_{y=a}$ is a decomposition of $f - c$, where
$\deg h_i|_{y=a} \ge 1$ for both $i$ because $\deg h_{3-i}|_{y=a} < \deg f$ for both $i$. \qedhere

\end{description}
\end{proof}

\noindent
In 3) of the lemma below, $i\le j$ is written instead of $i<j$ to include the case $n =1$,
where 2) and 3) of the lemma below are always satisfied (take $p = x_1$ and $q = 1$ in 2)).

\begin{lemma} \label{irred2}
Assume that $f, g \in \bar{K}[x]$ have degree at most $d$. Then for
\begin{enumerate}

\item[1)] $f - cg$ is reducible for at least $d^2$ values of $c \in \bar{K}$ and
          $\gcd\{f,g\} \in \bar{K}$,

\item[2)] there exist $p, q \in \bar{K}[x]$ such that $f \in \bar{K}[p,q]$ and
          $g \in \bar{K}[p^2,pq,q^2,p^3,p^2q,\allowbreak pq^2,q^3]$,

\item[3)] $\gcd\{\det \jac_{x_i,x_j}(f,g)\mid 1\le i\le j\le n\} \notin \bar{K}^{*}$.

\end{enumerate}
we have 1) $\Longrightarrow$ 2) $\Longrightarrow$ 3).
\end{lemma}

\begin{proof}
If $fg = 0$, then 1) $\Longrightarrow$ 2) follows, because 
$f, g \in \bar{K}$ if $fg = 0$ and $\gcd\{f,g\} \in \bar{K}$.
If $fg \ne 0$, then 1) $\Longrightarrow$ 2) follows from Corollary 2 of 
\cite[Th.\@ 37]{MR1770638}, because $f + x_{n+1} g$ is irreducible 
if $\gcd\{f,g\} \in \bar{K}$ and $fg \ne 0$. So 1) $\Longrightarrow$ 2)
is satisfied in any case.

If $\jac f$ and $\jac g$ are linearly dependent as row vectors over $\bar{K}(x)$, then the
formula in 3) equals zero, so assume the opposite. Then $f$ and $g$
are algebraically independent over $\bar{K}$. Suppose that 2) holds.
Then $p$ and $q$ are algebraically independent over $\bar{K}$ as well,
and by the chain rule
$$
\jac_{x_i,x_j}(f,g) = \jac_{p,q}(f,g) \cdot \jac_{x_i,x_j}(p,q)
$$
Hence $\det \jac_{p,q}(f,g)$ divides the formula in 3). Since both entries of
$\jac_{p,q} f$ are contained in $\bar{K}[p,q]$, and both entries of
$\jac_{p,q} g$ are linear combinations over $\bar{K}$ of terms $p^iq^j$ with $i+j\geq 1$, we have
$\det \jac_{p,q}(f,g) \in \bar{K}[p,q] \setminus \bar{K}$,
and 3) follows. This gives 2) $\Longrightarrow$ 3).
\end{proof}

\section{Irreducibility results for arbitrary Keller maps}

\begin{theorem} \label{irredlc}
Let $F \in K[x]^n$ be a Keller map of degree $d$.
Take $i \in \{1,2,\ldots,n,\allowbreak n+1\}$ and fix $\mu_1, \ldots,\mu_{i-1}, \mu_{i+1},
\ldots, \mu_{n+1} \in \bar{K}$. If $\mu_j \ne 0$ for some $j$ with $i \ne j \le n$,
then
$$
f := \mu_1 F_1 + \mu_2 F_2 + \cdots + \mu_n F_n + \mu_{n+1}
$$
is reducible over $\bar{K}$ for at most $d^2 - 1$ values of $\mu_i \in \bar{K}$.
\end{theorem}

\begin{proof}
Assume without loss of generality that $j = 1$.
Then $(f, F_2, \ldots, F_n)$ is a Keller map as well. By expansion of the Jacobian determinant
along the first row, we see that $\gcd\big\{\parder{}{x_1} f, \parder{}{x_2} f, \ldots,
\parder{}{x_n} f\big\} \in K^{*} \subseteq \bar{K}^{*}$. Hence the
case $i = n+1$ follows from 2) $\Longrightarrow$ 4) of Lemma \ref{irred}.

If $i \le n$, then expansion of the
Jacobian determinant along the first and the $i$-th row gives
$\gcd\{\det\jac_{x_k,x_l}(f,F_i)\mid 1\le k < l\le n\} \in K^{*} \subseteq \bar{K}^{*}$.
So the case $i \le n$ follows from 1) $\Longrightarrow$ 3) of Lemma \ref{irred2}.
\end{proof}

\begin{corollary}
Assume that $F \in K[x]^n$ is a Keller map.
Then there exists a $\lambda \in K^n$ and a $T \in \GL_n(K)$ such that the components of
both $F + \lambda$ and $T^{-1}F(Tx)$ are all irreducible over $\bar{K}$.
\end{corollary}

\begin{proof}
Notice that $\#K \ge d^2$ because $\chr K = 0$, where $d:=\deg F$.
The first claim follows from the case $i = n+1$ of Theorem \ref{irredlc} and the
second claim follows from the case $i \le n$ of Theorem \ref{irredlc} with
$\mu_{n+1} = 0$.
\end{proof}

\noindent
In \cite[Th.\@ 3]{MR1886939}, the authors proved additional properties for the $T \in \GL_n(K)$
when $n \ge 3$ and $K \subseteq \C$, namely that there exists a $T \in \GL_n(K)$ such that for every
$\lambda \in K^n$, every component of $T^{-1}F(Tx) + \lambda$ is irreducible over $\bar{K}$.

\begin{lemma} \label{mulem}
Assume $F \in K[x]^n$ is any polynomial map over $K$.
Let $d \ge 2$ be an integer and $\lambda \in K^n$. 
Then for all $h \in K[x,x_{n+1}]$, the map 
$G = (F - \lambda x_{n+1}^d, x_{n+1}, x_{n+2} - h)$ 
has the property that for all $\mu \in K^{n+3}$, either
$$
g := \mu_1 G_1 + \mu_2 G_2 + \cdots + \mu_{n+2} G_{n+2} + \mu_{n+3}
$$
is irreducible, or $\mu_{n+1} = \mu_{n+2} = 0$ and
$$
f := \mu_1 F_1 + \mu_2 F_2 + \cdots + \mu_n F_n + \mu_{n+3} = g,
$$
or we have $\mu_{n+2} = 0$ and $f \in \bar{K}[p^2,p^3]$ for some $p \in \bar{K}[x]$, in which case
$\jac f$ is not unimodular.
\end{lemma}

\begin{proof}
Assume that $h \in K[x,x_{n+1}]$ and that $g$ is reducible. Then $\mu_{n+2} = 0$
because otherwise $g$ would be a tame coordinate. Hence $g - f \in K[x_{n+1}]$. Since
the monomials of $g - f$ can only have degrees $1$ and $d$, we even have $g - f \in K[x_{n+1}]
\setminus K^{*}$. If $f = g$, then $\mu_{n+1} = 0$ because $\mu_{n+1} x_{n+1}$ is the difference
between the linear parts of $g$ and $f$.
If $f \ne g$, then by 1) $\Longrightarrow$ 3) of Lemma \ref{irred} (with $f - g$ instead of $g$),
$f \in \bar{K}[p^2,p^3]$ for some $p \in \bar{K}[x]$,
and 3) $\Longrightarrow$ 4) of Lemma \ref{irred} tells us $\jac f$ is not unimodular over $\bar{K}$
and hence neither over $K$.
\end{proof}

\begin{lemma} \label{mulem2}
Assume that $F \in K[x]^n$ is any polynomial map over $K$.
Let $\Lambda \in K[y]^n$ such that $y_1, \allowbreak y_2, \allowbreak
\ldots, y_n, \Lambda_1, \Lambda_2, \ldots, \Lambda_n, 1$ are linearly independent over $K$.
Then for all $H \in K[x,y]^n$, the map
$$
G := (F - \Lambda, y, z - H)
$$
has the property that for all $\mu \in K^{3n+1}$, either
$$
g := \mu_1 G_1 + \mu_2 G_2 + \cdots + \mu_{3n} G_{3n} + \mu_{3n+1}
$$
is irreducible, or $\mu_{2n+1} = \mu_{2n+2} = \cdots = \mu_{3n} = 0$ and for
$$
f := \mu_1 F_1 + \mu_2 F_2 + \cdots + \mu_n F_n + \mu_{3n+1},
$$
we have $f \in \bar{K}[p^2,p^3]$ for some $p \in \bar{K}[x]$, in which case $\jac f$ is not
unimodular.
\end{lemma}

\begin{proof}
Assume that $H \in K[x,y]^n$ and suppose that $g$ is reducible. Then $\mu_{2n+1} = \mu_{2n+2}
= \cdots = \mu_{3n} = 0$ because otherwise $g$ would be a tame coordinate.
Hence $g - f \in K[y]$. Since $1$ is linearly independent over $K$ of
$y_1, \allowbreak y_2, \allowbreak \ldots, y_n, \Lambda_1, \Lambda_2, \ldots, \allowbreak \Lambda_n$,
as opposed to $g - f$,
we even have $g - f \in K[y] \setminus K^{*}$. If $f = g$, then the linear independence over $K$
of $y_1, y_2, \ldots, y_n, \Lambda_1, \Lambda_2, \ldots, \Lambda_n$ tells us that $f = \mu_{3n+1}
\in \bar{K}[p^2,p^3]$ for any $p \in \bar{K}[x]$ and that $\jac f = (0^1~0^2~\cdots~0^n)$ is not
unimodular. The case $f \ne g$ follows in a similar manner as in the proof of Lemma \ref{mulem}.
\end{proof}

\begin{corollary} \label{mulem2col}
Assume that $F \in K[x]^n$ is a Keller map over $K$ and $d \ge 2$ be
an integer. Then for all $H \in K[x,y]^n$, the map
$$
G := (F - y^{*d}, y, z - H)
$$
has the property that
$$
\mu_1 G_1 + \mu_2 G_2 + \cdots + \mu_{3n} G_{3n} + \mu_{3n+1}
$$
is irreducible for all $\mu \in K^{3n+1}$ such that $\mu_i \ne 0$ for some $i \le 3n$.
\end{corollary}

\begin{proof}
Assume that $H \in K[x,y]^n$ and suppose that
$g := \mu_1 G_1 + \mu_2 G_2 + \cdots + \mu_{3n} G_{3n} + \mu_{3n+1}$ is
reducible and $\mu_i \ne 0$ for some $i \le 3n$. Since the linear parts of the components of
$G$ are linearly independent over $K$, we see that $g \ne 0$. Let
$f := \mu_1 F_1 + \mu_2 F_2 + \cdots + \mu_n F_n + \mu_{3n+1}$.
By Lemma \ref{mulem2}, we have $i \le 2n$.
Since $\deg g \ne 1$ by reducibility of $g$, we can even take $i \le n$.

Again by Lemma \ref{mulem2}, $\jac f$ is not unimodular, so
$(F_1, \ldots, F_{i-1}, f, F_{i+1}, \ldots, F_n)$ is not a Keller map.
This contradicts that $F$ is a Keller map, so $g$ is irreducible if $\mu_i \ne 0$ 
for some $i \le 3n$.
\end{proof}

\noindent
The proposition below can be used to show that $F_i - c$ is irreducible for all 
$c \in K$ in some additional cases. If we take for $v$ the first column of $T^{-1}$,
where $T$ is as in the proof of Proposition \ref{prop1}, then
$$
T_2 v = 0 \ne T_1 v \qquad \mbox{and} \qquad \jac h \cdot v = 0 \ne \jac f \cdot v
$$
where $h$ and $f$ are as in Proposition \ref{prop1}. 
So Proposition \ref{prop2} generalizes Proposition \ref{prop1} in some sense.
 
\begin{proposition} \label{prop2}
Assume that $f,h \in K[x]$ such that $\deg (f - h) = 1$.
If there exists a vector $v \in K^n$ such that 
$\jac h \cdot v = 0 \ne \jac f \cdot v$, then
$f$ is a tame coordinate.
\end{proposition}

\begin{proof}
Since $v \ne 0$, there exists a $T \in \GL_n(K)$ such that $v = T e_1$, where $e_1$ is the first standard
basis unit vector and hence $T e_1$ is the first column of $T$. Consequently,
$$
\jac (h(Tx)) \cdot e_1 = (\jac h)|_{x = Tx} \cdot v = 0 \ne (\jac f)|_{x = Tx} \cdot v = \jac (f(Tx)) \cdot e_1.
$$
It follows that $h(Tx) \in K[x_2, x_3,\ldots,x_n]$ and $f(Tx) \notin K[x_2, x_3,\ldots,x_n]$.
Since $\deg (f(Tx) - h(Tx)) = 1$, we see that $f(Tx) - h(Tx) - c x_1 \in K[x_2, x_3,\ldots,x_n]$ for some
$c \in K^{*}$. Hence $f(Tx) - c x_1 \in K[x_2, x_3,\ldots,x_n]$ and
$E := (c^{-1}f(Tx), x_2, \allowbreak x_3, \allowbreak \ldots, x_n)$ is an elementary invertible polynomial
map. Since $f$ is the first component of $c E(T^{-1} x)$, we see that $f$ is a tame coordinate.
\end{proof}

\section{Irreducibility results for cubic Keller maps without quadratic parts}

\begin{theorem} \label{irredth}
Assume $F = x + H$ is a polynomial map over $K$,
such that $H \in K[x]^n$ is cubic homogeneous and $\jac H$ is nilpotent. Say that besides
linear combinations of $1$ only, there are exactly $s \ge 1$ linear combinations of
$F_1, F_2, \ldots, F_n, 1$ which are reducible, if we do not count scalar multiples.

Then $s \le n-4$ and there exists a $T \in \GL_n(K)$ such that the first $s$ components of
$T^{-1}F(Tx)$ are reducible. In particular, $s$ is finite and the first $s$ components of
$T^{-1}F(Tx)$ are the only linear combinations of the components of $T^{-1}F(Tx)$ and $1$,
which are reducible and not a linear combination of $1$ only.
\end{theorem}

\begin{proof}
Notice that $F$ is a Keller map and therefore, $\jac F_i$ is unimodular for each $i$.
By 1) $\Longrightarrow$ 3) of Corollary \ref{irredcor}, all reducible linear combinations
of $F_1, F_2, \ldots, F_n, 1$ are already linear combinations of $F_1, F_2, \ldots, F_n$.
Replace $F$ by a linear conjugation of $F$ such that as many components of $F$ as possible
become reducible, say that exactly $t$ such components become reducible. Assume without
loss of generality that $F_1, F_2, \ldots, F_t$ are the reducible components of $F$.

It suffices to show that $t = s$ and 
\begin{equation} \label{tn4}
t \le n-4 
\end{equation}
We first show \eqref{tn4} by distinguishing $t > n - 4$ into three cases.
\begin{itemize}

\item \emph{$t > n - 4 \le 0$.} \\
Then $n \le 4$ and we have $s = t = 0$ on account of E. Hubbers' result that the JC
holds for $F$, see \cite{Hub94} or \cite[Cor.\@ 7.1.3]{MR1790619}. This contradicts $s \ge 1$.

\item \emph{$t > n-4 > 0$ and for each $i \le t$, there exists a $j \le t$ such that $H_j =
\lambda_j x_i x_j^2$ for some $\lambda_j \in K^{*}$.} \\
Notice that $j$ as above is unique for all $i \le t$, hence $i \mapsto j$ is a permutation
of $\{1,2,\ldots,t\}$, say with a cycle of length $k \le t$. Then we may assume without loss of generality
that
\begin{align*}
H_1 &= \lambda_1 x_k x_1^2,& H_2 &= \lambda_2 x_1 x_2^2,& H_3 &= \lambda_3 x_2 x_3^2,&
&\ldots,& H_k &= \lambda_k x_{k-1} x_k^2.
\end{align*}
The leading principal
minor determinant of size $k$ of $\jac H$ equals
$$
\big(2^k - (-1)^k\big) \lambda_1 x_1^2 \lambda_2 x_2^2 \cdots \lambda_k x_k^2,
$$
so the corresponding submatrix is not nilpotent.
But since $F_i \in K[x_1,x_2,\allowbreak \ldots,x_k]$ for all $i \le k$, the leading principal minor matrix of
size $k$ is nilpotent, because its $p$-th power is a submatrix of $(\jac H)^p$. We obtain a contradiction.

\item \emph{$t > n - 4 > 0$ and for some $i \le t$, there does not exist a $j \le t$ such that $H_j =
\lambda_j x_i x_j^2$ for some $\lambda_j \in K^{*}$.} \\
Notice that $(\jac H)|_{x_i=0}$ is nilpotent because $\jac H$ is nilpotent.
Since $F_i = x_i + H_i$ is reducible, it follows from 1) $\Longrightarrow$ 3) of Corollary
\ref{irredcor} that the $i$-th row
of $(\jac H)|_{x_i=0}$ is zero. Hence the principal minor matrix that we obtain from
$(\jac H)|_{x_i=0}$ by removing its $i$-th row and $i$-th column is nilpotent as well.
This minor matrix is equal to
$$
\jac_{x_1,\ldots,x_{i-1},x_{i+1},\ldots,x_n}
\big(H_1|_{x_i=0},\ldots,H_{i-1}|_{x_i=0},H_{i+1}|_{x_i=0},\ldots,H_n|_{x_i=0}\big).
$$
By 1) $\Longrightarrow$ 3) of Corollary \ref{irredcor} we see that $x_j^2 \mid H_j$
for each $j \le t$. But by assumption on $i$, we have $x_i x_j^2 \nmid H_j$ for each $j \le t$.
Hence by cubic homogeneity of
$H_j$, we have $H_j|_{x_i=0} \ne 0$ for each $j \le t$ except $j = i$. So $\deg F_j = \deg F_j|_{x_i = 0}$
and $\deg g = \deg g|_{x_i=0}$ for every $g \mid F_j$ and each $j \le t$ except $j = i$.

As a consequence, $F_1|_{x_i=0}, \ldots, F_{i-1}|_{x_i=0}, F_{i+1}|_{x_i=0}, \ldots, \allowbreak
F_t|_{x_i=0}$ are all reducible. 
Make $\hat{F}$ from $F|_{x_i=0}$ by removing the $i$-th
component and substituting $x_{i+1} = x_i, x_{i+2} = x_{i+1}, \ldots, x_n = x_{n-1}$, in that order.
Then $\hat{F}$ has $n-1$ components, of which the first $t-1$ are reducible. 
By induction on $n$, it follows that $(t-1) \le (n-1) - 4$, so $t \le n-4$.

\end{itemize}
So it remains to show that $t = s$. Suppose therefore that $t \ne s$. Then $t < s$ and
by maximality of $t$, we cannot get the first $t+1$ components of $F$ reducible by way of conjugation.
Hence all linear combinations of components of $F$
which are reducible are already linear combinations of $F_1, F_2, \ldots, F_t$.
Since $t < s$, there exists a linear combination of the form $\mu_1 F_1 + \mu_2 F_2 + \cdots +
\mu_t F_t$ which is reducible, such that $\mu_i \ne 0$ for at least two $i$'s.
Hence we may assume that there is a reducible linear combination
of the form $\mu_1 F_1 + \mu_2 F_2 + \cdots + \mu_r F_r$ with $\mu_1 \mu_2 \cdots \mu_r \ne 0$,
where $2 \le r \le t$.

Assume first that there exist $i \le r < k$ such that $\parder{}{x_k} H_i \ne 0$.
On account of 1) $\Longrightarrow$ 3) of Corollary \ref{irredcor}, for each $j \le r$,
we have $H_j = x_j^2 g_j$ for some linear
form $g_j$. Therefore,
\begin{equation} \label{kdif}
\parder{}{x_k} (\mu_1 H_1 + \mu_2 H_2 + \cdots + \mu_r H_r)
\end{equation}
is a nontrivial $K$-linear combination of $x_1^2, x_2^2, \ldots, x_r^2$.
Since the coefficient of $x_1 x_2$ in
$(\mu_1 x_1 + \mu_2 x_2 + \cdots + \mu_r x_r)^2$ is $2 \mu_1 \mu_2 \ne 0$,
$(\mu_1 x_1 + \mu_2 x_2 + \cdots + \mu_r x_r)^2$ does not divide \eqref{kdif},
and neither divides $\mu_1 H_1 + \mu_2 H_2 + \cdots + \mu_r H_r$.
Now 1) $\Longrightarrow$ 3) of Corollary \ref{irredcor} tells us that the
Jacobian of $\mu_1 F_1 + \mu_2 F_2 + \cdots + \mu_r F_r$ is not unimodular.
Hence $(\mu_1 F_1 + \mu_2 F_2 + \cdots + \mu_r F_r, F_2, F_3, \ldots, F_n)$ is not a Keller map
and neither is $F$. This contradicts that $\jac H$ is nilpotent.

Assume next that $\parder{}{x_k} H_i = 0$ for all $i \le r < k$.
Then $H_i \in K[x_1,x_2,\ldots,x_r]$ for all $i \le r$ and the leading principal
minor matrix of size $r$ of $\jac H$ is nilpotent because its $p$-th power is a submatrix of
$(\jac H)^p$. Hence the map $\tilde{F}$ which consists of the first $r$ components of 
$F$ satisfies the conditions on $F$ of this theorem with $n = r$. On account of
\eqref{tn4}, at most $r - 4$ components of $\tilde{F}$ can be reducible. 
This contradicts that all components of $\tilde{F}$ are reducible.
\end{proof}

\begin{corollary} \label{lambdacol}
Assume that $F \in K[x]^n$ is a cubic Keller map over $K$ without quadratic part. Then there exists a
$\lambda \in K^n$ such that for all $h \in K[x,x_{n+1}]$, the map
$$
G := (F - \lambda x_{n+1}^3, x_{n+1}, x_{n+2} - h)
$$
has the property that
$$
\mu_1 G_1 + \mu_2 G_2 + \cdots + \mu_{n+2} G_{n+2} + \mu_{n+3}
$$
is irreducible for all $\mu \in K^{n+3}$ such that $\mu_i \ne 0$ for some $i \le n+2$.
\end{corollary}

\begin{proof}
Take $T$ as in Theorem \ref{irredth} and let $\lambda$ be the sum of the columns
of $T$. Assume that
$g := \mu_1 G_1 + \mu_2 G_2 + \cdots + \mu_{n+2} G_{n+2} + \mu_{n+3}$ is
reducible for some $\mu \in K^{n+3}$ such that $\mu_i \ne 0$ for some $i \le n + 2$.
If $f := \mu_1 F_1 + \mu_2 F_2 + \cdots + \mu_n F_n + \mu_{n+3} = g$, then $f(Tx) = g(Tx)$
is reducible as well, and by Theorem \ref{irredth}, $(\mu_1~\mu_2~\cdots~\mu_n)$ is
$c$ times a row of $T^{-1}$ for some $c \in K^{*}$. By the choice of $\lambda$,
we have $g = f - (\mu_1~\mu_2~\cdots~\mu_n) \lambda x_{n+1}^3 = f - c x_{n+1}^3$ in this case,
which contradicts $f = g$.

So $f \ne g$ and by Lemma \ref{mulem}, $\jac f$ is not unimodular and $\mu_{n+2} = 0$.
Since $\mu_{n+1} G_{n+1} = \mu_{n+1} x_{n+1}$ is irreducible if $\mu_{n+1} \neq 0$, we can choose
$i \le n$. Hence $(F_1, \ldots, F_{i-1},\allowbreak f,F_{i+1},\ldots,F_n)$ is a Keller map
as well as $F$. This contradicts that $\jac f$ is not unimodular, so $g$ is irreducible.
\end{proof}

\section{Irreducibility results for symmetric Keller maps}

\begin{lemma} \label{symdiag}
Assume $F$ is a Keller map in dimension $n$ over $K$, such that
$\jac F$ is symmetric. If $F_i$ is of the form $c' x_i + x_i^2 h + c$
for some $c,c' \in K$ and some $h \in K[x]$, then $h = 0$.
\end{lemma}

\begin{proof}
Suppose that  $F_i$ is of the form $c' x_i + x_i^2 h + c$. Then the constant part
of $\jac F_i$ is of the form
$$
(0^1 ~ \cdots ~ 0^{i-1} ~ c' ~ 0^{i+1} ~ \cdots ~ 0^n).
$$
By expansion along the $i$-th row of $\jac F$, we see that
$p = \det \jac F$ can be expressed as $p = p_0 + (\jac F)_{ii} p_1$, where
both $p_0$ and $p_1$ are polynomials in the entries of $\jac F$ except $(\jac F)_{ii}$, in
such a way that the constant part with respect to $x$ of $p_0$ is zero. We will use this
to prove that $h = 0$.

So let us assume that $h \ne 0$, say that $h = x_i^{r-2} \tilde{h}$, where $x_i \nmid \tilde{h}$.
Then $F_i = c' x_i + x_i^r \tilde{h} + c$. From Poincar\'e's lemma, it follows that
$F = \grad f$ for some $f \in K[x]$, so
$\parder{}{x_i} f = F_i$. If $\tilde{h}$ has no monomials that are divisible by $x_i$, then
\begin{align*}
f &= f|_{x_i = 0} \cdot 1 + c \cdot x_i + \tfrac12 c' \cdot x_i^2
+ \tfrac1{r+1} \tilde{h} \cdot x_i^{r+1}. \\
\intertext{In the general case, we can turn out monomials of $\tilde{h}$
that are divisible by $x_i$ by reducing $\tilde{h} \cdot x_i^{r+1}$ modulo $x_i^{r+2}$, and we have}
f \bmod{x_i^{r+2}} &= f|_{x_i = 0} \cdot 1 + c \cdot x_i + \tfrac12 c' \cdot x_i^2
+ \tfrac1{r+1} \tilde{h}|_{x_i=0} \cdot x_i^{r+1}.
\end{align*}
where $\tilde{h}|_{x_i=0} \ne 0$.
Since $r \ge 2$ by definition, it follows that $(\jac F)_{ii} = \parder[2]{}{x_i} f$ is the only 
entry of the matrix $\jac F = \hess f$ with monomials of degree between $1$ and $r-1$ inclusive in $x_i$,
and those monomials add up to $r \tilde{h}|_{x_i=0} \cdot x_i^{r-1} \ne 0$. Hence $p_0$ and $p_1$ do
not have monomials of degree between $1$ and $r-1$ inclusive in $x_i$ either. Furthermore, there
exists a term $t$ of degree between $1$ and $r-1$ inclusive in $x_i$ whose coefficient
in $(\jac F)_{ii}$ is nonzero. We can choose $t$ of minimum degree, 
so that $t$ is not divisible by any other such term of $(\jac F)_{ii}$.

Since the coefficient of $1$ in $p_0$ is zero, the coefficient of $1$ in $p = \det \jac F$, which is 
nonzero because of the Keller condition on $F$, is equal to the coefficient of $1$ in 
$(\jac F)_{ii} \cdot p_1|_{x=0}$. The coefficient of $t$ in $(\jac F)_{ii} \cdot p_1|_{x=0}$
is nonzero as well, because $p_1|_{x=0} \in K^{*}$ along with the coefficient of $1$ in 
$(\jac F)_{ii} \cdot p_1|_{x=0}$. 

So if we can show that the coefficient of $t$ in 
$p = \det \jac F$ is equal to that in $(\jac F)_{ii} \cdot p_1|_{x=0}$, then we 
have a contradiction with the Keller condition on $F$, which gives us the conclusion that $h = 0$. 
Indeed, since $p_0$ do not have monomials of degree between $1$ and $r-1$ inclusive, the
coefficient of $t$ in $p$ is equal to that in $(\jac F)_{ii} \cdot p_1$. Now, $p_1$ does
not have such monomials either, therefore our choice of $t$ tells us that coefficient of 
$t$ in $p$ is equal to that in $(\jac F)_{ii} \cdot p_1|_{x=0}$. 
\end{proof}

\noindent
From now on in this section, we shall write $G_i^{(k)}$ for the homogeneous part of degree $k$ of a 
polynomial $G_i$ (which is the $i$-th component of a polynomial map $G$), or zero if $G_i$ has no such part.
We did this already for $k = 1$ in Theorem \ref{bakext}. Furthermore, we did something similar for
the polynomials $f$ and $g$ in section \ref{irlemu}.

\begin{corollary} \label{symdiagcol}
Assume $G$ is a Keller map in dimension $n$ over $K$, such that $\jac G$ is symmetric.
Suppose that $\parder{}{x_i} G_i^{(1)} \ne 0$.

If $G_i$ is of the form $G_i^{(1)} + \big(G_i^{(1)}\big)^2 h + c$
for some $c \in K$ and $h \in K[x]$, then $h = 0$ and hence $\deg G_i = 1$.
\end{corollary}

\begin{proof}
Take $T \in \GL_n(K)$ such that $T$ corresponds to the identity matrix $I_n$ except
for the $i$-th row, for which we take $\big(\parder{}{x_1} G_i^{(1)} ~ \parder{}{x_2} G_i^{(1)}
~ \cdots ~ \parder{}{x_n} G_i^{(1)}\big)$. From Poincar\'e's lemma, it follows that
$G = \grad g$ for some $g \in K[x]$. Next, define $f := g(T^{-1}x)$ and
$F := (T^{-1})\tp G(T^{-1}x)$. Since $(T^{-1})\tp$ corresponds to the identity matrix $I_n$ except
for the $i$-th column, we have $F_i = ((T^{-1})\tp)_i G(T^{-1}x) = (T^{-1})_{ii} G_i(T^{-1}x)$.

By definition of $T$, the linear part of $G_i$ is equal to $(Tx)_i$, thus the linear part of
$G_i(T^{-1}x)$ is equal to $G_i^{(1)}(T^{-1}x) = x_i$.
Hence $F_i = (T^{-1})_{ii} G_i(T^{-1}x) = (T^{-1})_{ii} \big(x_i + x_i^2 h(T^{-1}x) + c\big)$.
Furthermore,
$$
\grad f = (\jac f)\tp = \big(\jac g(T^{-1}x)\big)\tp = \big(G\tp|_{x=T^{-1}x}\cdot T^{-1}\big)\tp =
(T^{-1})\tp G(T^{-1}x) = F,
$$
so $\jac F$ is symmetric, and we have $h(T^{-1}x) = 0$ on account of Lemma \ref{symdiag}.
This gives the desired result.
\end{proof}

\begin{theorem} \label{symth}
Assume $G$ is a Keller map in dimension $n$ over $K$, such that $\jac G$ is symmetric.
Suppose that $\parder{}{x_i} G_i^{(1)} \ne 0$ and $G_i^{(1)} \mid G_i^{(k)}$ for all $k \in
\{1,2,\ldots,\allowbreak d-1\}$.

If $G_i - c$ has a divisor of degree less than $d$ with trivial constant part
for some $c \in K$, then $\deg G_i = 1$.
\end{theorem}

\begin{proof}
Suppose that $g \mid G_i - c$, $c \in K$, $\deg g < d$ and $g(0) = 0$. Say that $G_i - c = g h$.
Since $g(0) = 0$, we have $h(0) \mid G_i^{(1)} \ne 0$. From Lemma
\ref{x_1|flem} with $g^{*} = G_i^{(1)}$, we obtain that $G_i^{(1)} \mid g \mid G_i - c$.
Hence we may assume that $g = G_i^{(1)}$.

By the Keller condition, $G_i$ is nonsingular over $\bar{K}$, which gives
that $G_i$ is of the form of Corollary \ref{symdiagcol} above. On account of that corollary,
$\deg G_i = 1$.
\end{proof}

\begin{lemma} \label{irredAy}
Let $f$ be of the form $g_0 + g_1 y_1 + g_2 y_2 + \cdots + g_n y_n \neq g_0$,
where $g_i \in A$ for all $i$ for some unique factorization domain $A$. Then $f$ is irreducible,
if and only if $\gcd\{g_0,g_1,g_2,\ldots,g_n\} \in A^{*}$.
\end{lemma}

\begin{proof}
If $f$ decomposes in two factors, then one of the factors is constant with respect
to $y$. Now the conclusion follows easily.
\end{proof}

\begin{lemma} \label{symm}
Assume that $F \in K[x]^n$ is a polynomial map over $K$ and
$f \in K[x]$. Set $G := \grad_{x,y} (f + y\tp F)$. Then $\det \jac_{x,y} G = (-1)^n (\det \jac F)^2$
and $F$ is invertible, if and only if $G$ is invertible. Furthermore,
\begin{equation} \label{irredG}
\mu_1 G_1 + \mu_2 G_2 + \cdots + \mu_n G_n + \cdots + \mu_{2n} G_{2n} + \mu_{2n+1}
\end{equation}
is irreducible for all $\mu \in K^{2n+1}$ such that $\mu_i \ne 0$ for some $i \le n$ if $F$ is a
Keller map.
\end{lemma}

\begin{proof}
Notice that $\jac_{x,y} G$ is of the form
$$
\jac_{x,y} G = \hess_{x,y} (f + y\tp F) = \left( \begin{array}{cc} * & (\jac F)\tp \\
                           (\jac F) & 0 \end{array} \right),
$$
whence $\det \jac_{x,y} G = (-1)^n (\det \jac F)^2$. Consequently, both $F$ and $G$ are
Keller maps if one of them is, which we assume from now on.
Notice that $G = \big((\jac F)\tp y + \grad f, F\big)$ by definition. Hence we have
$$
G\Big(x,\big((\jac F)\tp\big)^{-1}(y-\grad f)\Big) = (y,F),
$$
and we see that that $G$ is invertible, if and only if $F$ is invertible.

Suppose that $\mu \in K^{2n+1}$ such that $\mu_i \ne 0$
for some $i \le n$. Then \eqref{irredG} is of the form $g_0 + g_1 y_1 + g_2 y_2 +
\cdots + g_n y_n$ with $g_i \in K[x]$ for all $i$. More precisely,
$$
(g_1 ~ g_2 ~ \cdots ~ g_n) = \jac_y \Big((\mu_1 ~ \mu_2 ~ \cdots ~ \mu_n) \big((\jac F)\tp y + \grad f\big)\Big)
$$
is a nontrivial linear combination of the rows of $(\jac F)\tp$, so $(g_1 ~ g_2 ~ \cdots ~ g_n)$ is
unimodular. In particular, $\gcd\{g_0,g_1,g_2,\ldots,g_n\} \in K^{*}$. Hence \eqref{irredG} is
irreducible on account of Lemma \ref{irredAy}.
\end{proof}

\noindent
Acknowledgment: The authors are very grateful to the referees who give some valuable advices. 

\bibliographystyle{weakjc}
\nocite{MR0592226E} %\cite{MR0592226E}
\bibliography{weakjc}

\end{document}